\documentclass[10pt,a4paper]{article}
\usepackage{amsmath}
\usepackage{amsfonts}
\usepackage{amssymb}
\usepackage{graphicx}

\newtheorem{thm}{Theorem}[section]
\newtheorem{lemma}[thm]{Lemma}

\newtheorem{pro}[thm]{Proposition}

\newenvironment{proof} {\par \noindent \textbf{Proof: }}{\QED \par \bigskip \par}
\newcommand{\QED}{\hfill$\square$}
\newcommand{\rz}{\vspace{0.1cm}}

\title {
\bf{On the clique number of integral circulant graphs}
\thanks{The authors gratefully acknowledge support from
    research projects 144011 and 144007 of the Serbian Ministry of
    Science and Environmental Protection.}
}

\author {
{\large Milan Ba\v si\' c} \\
{\em Faculty of Sciences and Mathematics, University of Ni\v s, Serbia} \\
{e-mail: \texttt{basic\_milan@yahoo.com}} \and
{\large Aleksandar Ili\' c} \footnotemark [3] \\
{\em Faculty of Sciences and Mathematics, University of Ni\v s, Serbia} \\
{e-mail: \texttt{aleksandari@gmail.com} } }

\begin{document}

\maketitle

\begin{abstract}
The concept of gcd-graphs is introduced by Klotz and Sander, which
arises as a generalization of unitary Cayley graphs. The gcd-graph
$X_n (d_1,...,d_k)$ has vertices $0,1,\ldots,n-1$, and two vertices
$x$ and $y$ are adjacent iff $\gcd(x-y,n)\in D =
\{d_1,d_2,...,d_k\}$. These graphs are exactly the same as circulant
graphs with integral eigenvalues characterized by So. In this paper
we deal with the clique number of integral circulant graphs and
investigate the conjecture proposed in \cite{klotz07} that clique
number divides the number of vertices in the graph $X_n (D)$. We
completely solve the problem of finding clique number for integral
circulant graphs with exactly one and two divisors. For $k \geqslant
3$, we construct a family of counterexamples and disprove the conjecture
in this case.
\end{abstract}

\footnotetext[3] { Corresponding author. If possible, send your
correspondence via e-mail. Otherwise, the postal address is:
Department of Mathematics and Informatics, Faculty of Sciences and
Mathematics, Vi\v segradska 33, 18000 Ni\v s, Serbia }

\section{Introduction}

Integral circulant graphs have been proposed as potential candidates
for modeling quantum spin networks that might enable the perfect
state transfer between antipodal sites in a network. Motivated by
this, Saxena, Severini and Shraplinski \cite{saxena07} studied some
properties of integral circulant graphs --- bounds for number of
vertices and diameter, bipartiteness and perfect state transfer.
Stevanovi\'c, Petkovi\'c and Ba\v si\'c \cite{stevanovic08} improved
the previous upper bound for diameter and showed that the diameter
of these graphs is at most $O(\ln \ln n)$.  Circulant graphs are
important class of interconnection networks in parallel and
distributed computing (see \cite{hwang03}). \rz

Various properties of unitary Cayley graphs as a subclass of
integral circulant graphs were investigated in some recent papers.
In the work of Berrizbeitia and Giudici \cite{berrizbeitia04} and in
the later paper of Fuchs \cite{fuchs05}, some lower and upper
bounds for the longest induced cycles were given. Stevanovi\'c, Petkovi\'c and
Ba\v si\'c \cite{stevanovic08a} established a characterization of
integral circulant graphs which allows perfect state transfer and
proved that there is no perfect state transfer in the class of
unitary Cayle graphs except for hypercubes $K_2$ and $C_4$. Klotz
and Sander \cite{klotz07} determined the diameter, clique number,
chromatic number and eigenvalues of unitary Cayley graphs. The
latter group of authors proposed a generalization of unitary Cayley
graphs named {\it gcd-graphs} and proved that they have to be
integral. Integral circulant graphs were characterized by So
\cite{so06} --- a circulant graph is integral if and only if it is a
gcd-graph. This is the solution to the second proposed question in
\cite{klotz07}. \rz

Motivated by the third concluding problem in \cite{klotz07}, we
investigate the clique number of integral circulant graphs $X_n
(D)$, where $D = \{d_1, d_2, \ldots, d_k\}$ and the numbers $d_i$ are
proper divisors of $n$. In Section 2 we extend the result of clique
number and chromatic number for unitary Cayley graphs that are not
connected. In Section 3 we completely characterize the clique number
for integral circulant graphs with exactly two divisors $X_2 (d_1,
d_2)$. In previous cases when $k = 1$ or $k = 2$, the conjecture
that the clique number of a graph $X_n (d_1, d_2, \ldots, d_k)$ must
divide $n$ is supported by Lemma \ref{unconnected unitary} and
Theorem \ref{main k=2}. In Section 4 we refute the conjecture for $k
\geqslant 3$ by constructing a class of counterexamples for $k = 3$
and $k = 4$. In Section 5 we propose a simple lower and upper bound
for $\omega (X_n (d_1, d_2, \ldots, d_k))$, where $k$ is an
arbitrary natural number.

\section{Preliminaries}

Let us recall that for a positive integer $n$ and subset $S
\subseteq \{0, 1, 2, \ldots, n - 1\}$, the circulant graph $G(n, S)$
is the graph with $n$ vertices, labeled with integers modulo $n$,
such that each vertex $i$ is adjacent to $|S|$ other vertices $\{ i
+ s \pmod n \ | \ s \in S\}$. The set $S$ is called a symbol of
$G(n, S)$. As we will consider only undirected graphs, we assume
that $s \in S$ if and only if $n - s \in S$, and therefore the
vertex $i$ is adjacent to vertices $i \pm s \pmod n$ for each $s \in
S$. \rz

Recently, So \cite{so06} has characterized integral circulant
graphs. Let
$$ G_n(d) = \{ k\ | \ gcd(k, n)=d, \ 1 \leq k < n \} $$
be the set of all positive integers less than $n$ having the same
greatest common divisor $d$ with $n$. Let $D_n$ be the set of
positive divisors $d$ of $n$, with $d \leqslant \frac{n}{2}$.
\begin{thm}[\cite{so06}]
    A circulant graph $G(n,S)$ is integral if and only if
    $$ S = \bigcup_{d \in D} G_n(d) $$
    for some set of divisors $D \subseteq D_n$.
    \label{so}
\end{thm}
Let $\Gamma$ be a multiplicative group with identity $e$. For
$S\subset \Gamma$, $e\not\in S$ and $S^{-1} = \{s^{-1}\ |\ s\in
S\}=S$, Cayley graph $X = Cay(\Gamma,S)$ is the undirected graph
having vertex set $V(X)=\Gamma$ and edge set $E(X) = \{\{a,b\}\ |\
ab^{-1}\in S\}$. For a positive integer $n > 1$ the unitary Cayley
graph $X_n = Cay(Z_n, \ U_n)$ is defined by the additive group of
the ring $Z_n$ of integers modulo $n$ and the multiplicative group
$U_n = Z_n^{*}$ of its units. \rz

Let $D$ be a set of positive, proper divisors of the integer $n >
1$. Define the gcd-graph $X_n(D)$ to have vertex set $Z_n = \{0, 1,
\ldots, n - 1 \}$ and edge set
$$E (X_n (D)) = \left \{ \{a, b \} \mid a, b \in Z_n, \ gcd (a - b, n) \in D \right \}.$$

If $D = \{d_1, d_2, \ldots, d_k \}$, then we also write $X_n (D) =
X_n (d_1, d_2, \ldots, d_k )$; in particular $X_n (1) = X_n$.
Throughout the paper, we let $n = p_1^{\alpha_1} p_2^{\alpha_2}
\cdot \ldots \cdot p_k^{\alpha_k}$, where $p_1 < p_2 < \ldots < p_k$
are distinct primes, and $\alpha_i \geqslant 1$. Also $f (n)$
represents the smallest prime divisor of $n$. By Theorem \ref{so} we
obtain that integral circulant graphs are Cayley graphs of the
additive group of $Z_n$ with respect to the Cayley set $S =
\bigcup_{d\in D}G_n(d)$ and thus they are gcd-graphs. From Corollary
4.2 in \cite{hwang03}, the graph $X_n(D)$ is connected if and only
if $gcd(d_1,d_2,\ldots,d_k)=1$.

\begin{thm}
\label{d components}
If $d_1, d_2, \ldots, d_k$ are divisors of $n$, such that
the greatest common divisor $gcd (d_1, d_2, \ldots, d_k)$ equals
$d$, then the graph $X_n (d_1, d_2, \ldots, d_k)$ has exactly
$d$ isomorphic connected components of the form
$X_{n / d} (\frac{d_1}{d}, \frac{d_2}{d}, \ldots, \frac{d_k}{d})$.
\end{thm}

\begin{proof}
If $gcd (d_1, d_2, \ldots, d_k) = d > 1$, in the graph $X_n (D)$ there are at least $d$
connected components. We will prove that the subgraph induced by vertices $\{ r,
d + r, 2d + r, \ldots, (\frac{n}{d} - 1) \cdot d + r\}$ is
connected, by constructing a path from vertex $r$ to every other vertex
in this component.\rz

The B\'{e}zout's identity states that for integers $a$
and $b$ one can find integers $x$ and $y$, such that $ax + by =
gcd (a, b)$. By induction, we will prove that there are integers
$x_1, x_2, \ldots, x_k$ such that $x_1 \cdot d_1 + x_2 \cdot d_2
+ \ldots + x_k \cdot d_k = d$. For $k > 2$, we can find integers $y_1,
y_2, \ldots y_{k - 1}$ such that $y_1 \cdot d_1 + y_2 \cdot d_2
+ \ldots + y_{k - 1} \cdot d_{k - 1} = gcd (d_1, d_2, \ldots,
d_{k - 1})$. Applying B\'{e}zout's identity on numbers $gcd (d_1, d_2, \ldots,
d_{k - 1})$ and $d_k$, it follows that
$$d = x \cdot gcd (d_1, d_2, \ldots, d_{k - 1}) + y \cdot d_k = x y_1 \cdot d_1 + x y_2 \cdot d_2
+ \ldots + x y_{k - 1} \cdot d_{k - 1} + y \cdot d_k.$$
Furthermore, for $i=1,2,\ldots, k$ and $j=0,1,\ldots, n-1$, let
$$H_n(d_i)=\{ h \mid 0 \leqslant h < n, \ h \equiv 0 \pmod {d_i} \} \subseteq Z_{n}, $$
and let $j+H_n(d_i)$ denote the subgraph of $X_n (D)$ with the vertex set
$\{ j + h \mid h \in H_n(d_i) \}$. Two vertices $j + h_1$
and $j + h_2$ are adjacent if $h_2-h_1\in G_n(d_i)$.
Thus, from vertex $r$ we can walk to every vertex with label $r
+ k \cdot d$, where $0 \leqslant k \leqslant \frac{n}{d}$, passing through
subgraphs $H_n (d_1), H_n (d_2),\ldots, H_n (d_k)$ consecutively.
\end{proof}

In \cite{klotz07} authors proved the following result for unitary
Cayley graphs.

\begin{thm}
If $D = \{ 1 \}$ then $\chi (X_n) = \omega (X_n) = f(n)$.
\label{divisorOne}
\end{thm}

Consider the set $D = \{ d \}$, where $d \geqslant 1$ is a divisor
of $n$. The graph $X_n (d)$ has $d$ connected components - the
residue classes modulo $d$ in $Z_n = \{0, 1, 2, \ldots, n - 1 \}$.
The degree of every vertex is $\phi (\frac{n}{d})$ where $\phi (n)$
denotes the Euler phi function.

\begin{lemma}
\label{unconnected unitary}
    For the gcd-graph $X_n (d)$ it holds that:
    $$\chi (X_n (d)) = \omega (X_n (d)) = f \left ( \frac{n}{d} \right ).$$
\end{lemma}

\begin{proof}
Let $p = f (\frac{n}{d})$ be the smallest prime divisor of $\frac n
d$. The vertices $0, d, 2d, \ldots, (p - 1) d$ induce a clique in
the graph $X_n (d)$, because the greatest common divisor of $d \cdot
(a - b)$ and $n$ equals $d$, for $0 < a, b < p$. Therefore, we have
inequality $\chi (X_n (d)) \geqslant \omega (X_n (d)) \geqslant p$.
\rz

On the other hand, consider the component with the vertices $r, d+r,
2d+r, \ldots, (\frac{n}{d} - 1) d+r$, for some $r\in Z_d$. Two
vertices $d \cdot a+r$ and $d \cdot b+r$ are adjacent if and only
if $gcd (a - b, \frac{n}{d}) = 1$, which is evidently true because
$|a - b| < \frac{n}{d}$. From Theorem \ref{divisorOne}, we get that
chromatic number in such component is the least prime dividing
$\frac{n}{d}$. The same observation holds for every residue class
modulo $d$, by Theorem \ref{d components}. Thus, the chromatic
number and the clique number in $X_n (d)$ are equal to $f
(\frac{n}{d})$.
\end{proof}

\section{Clique number for $k = 2$}

Let $D$ be a two element set $D = \{ d_1, d_2 \}$, where $d_1
> d_2$. Let $Q$ be the set of all prime divisors of $n$ that does not divide
$d$. The main result of this section is following theorem.

\begin{thm}
In the graph $X_n (d_1, d_2)$ we have:
\begin{eqnarray*}
\omega \left( X_n (d_1, d_2) \right) = \left\{
\begin{array}{rl}
\min \ \left ( \min_{p \in Q} p, f \left( n \right) \cdot f \left ( \frac
{n}{d}\right) \right), & \mbox { if $d_2 = 1$, } \\
\omega (X_{\frac{n}{d_2}} (1, \frac{d_1}{d_2})), & \mbox{ if $d_2 \mid d_1$ and $d_2 > 1$, } \\
\max \left ( f \left( \frac {n}{d_1}
\right), f \left(\frac {n}{d_2} \right) \right ), & \mbox{ otherwise. }\\
\end{array} \right.
\end{eqnarray*}
\end{thm}

According to the definition, the edge set of $X_n (d_1, d_2)$
is the union of the edge sets of graphs $X_n(d_1)$ and $X_n
(d_2)$. We color the edges of the graph $X_n (d_1, d_2)$ with two
colors: edge $\{a, b \}$ is blue if $gcd (a - b, n) = d_1$ and red
if $gcd (a - b, n) = d_2$. Therefore by Lemma \ref{unconnected
unitary},

\begin{equation}
\omega \left( X_n (d_1, d_2) \right) \geqslant \max \left ( f \left(
\frac {n}{d_1} \right), f \left(\frac {n}{d_2} \right) \right ).
\label{lower bound}
\end{equation}

\subsection{ Case $1 \in D$ }

Assume that $D = \{1, d \}$, where $d = p_1^{\beta_1}
p_2^{\beta_2} \cdot \ldots \cdot p_k^{\beta_k}$. Let $i$ be the
first index such that $\beta_i < \alpha_i$ i.e. $f (\frac{n}{d}) =
p_i$. By (\ref{lower bound}), we know that $\omega (X_n (1, d))
\geqslant p_i$.

\begin{lemma}
\label{inequality p1pi} In the graph $X_n (1, d)$ we have
$$\omega (X_n (1, d)) \leqslant f(n) \cdot f \left( \frac{n}{d} \right).$$
\end{lemma}

\begin{proof}
Color the edges of the graph $X_n (1, d)$ with two colors. Let blue
edges be those with $gcd (a - b, n) = d$ and red edges those with
$gcd (a - b, n) = 1$. If we have two adjacent blue edges $(a, b)$
and $(a, c)$ then edge joining the vertices $b$ and $c$ must be blue,
if exists. This follows from the fact that if $d \mid a - b$ and $d
\mid a - c$, then $d$ must divide $gcd (b - c, n)$. \rz

Thus, any two maximal cliques composed only of blue edges are vertex disjoint.
Now let $K_1, K_2, \ldots, K_x$ be all the maximal cliques of blue edges in a maximal
clique $C^\ast$ of $X_n(1, d)$. Then any edge joining two vertices in different cliques
$K_i$ and $K_j$ is red. Furthermore, any vertex in $V (C^\ast) \setminus ( V (K_1) \cup V (K_2)
\cup \ldots \cup V (K_x) )$ does not belong to any blue clique, which implies that
$V (C^\ast) \setminus ( V (K_1) \cup V (K_2) \cup \ldots \cup V (K_x) )$
induces a clique composed only of red edges and is denoted by C.
By Lemma \ref{unconnected unitary} the order of $K_i$ ($i\in \{1, 2, \ldots, x\}$)
is at most $p_i = f(\frac{n}{d})$ and the order of $C$ is at most $p_1$. \rz


Let $y$ be the size of clique $C$. If we choose one vertex
from each clique with blue edges, then these $x$
vertices with $y$ vertices from clique $C$ form a clique with red
edges in the graph $X_n (1, d)$. Therefore, $x + y \leqslant p_1$.
The size of every blue clique is bounded by $p_i$. The number of
vertices in the maximal clique of $X_n (1, d)$ is
$$
\sum_{j=1}^{x}|K_j|+|C|\leq x\cdot p_i+y=x\cdot (p_i-1)+(x+y)\leq
p_1\cdot (p_i-1)+p_1=p_i\cdot p_1,
$$
which means that the size of maximal clique is less than or equal to
$f (n) \cdot f (\frac{n}{d})$.
\end{proof}

\begin{figure}[h]
  \center
  \includegraphics [width = 5cm]{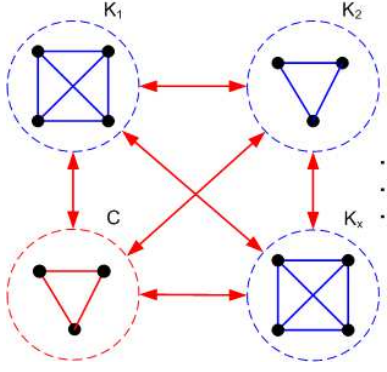}
  \caption { \textit{The maximal clique in graph $X_n (1, d)$}}
\end{figure}

Let $R$ be the set of all prime divisors of $d$ which also
divide $\frac{n}{d}$. Let $q$ be the least prime number from the set
$Q$, if exists. Let number $M$ be the product of all primes in both
sets $R$ and $Q$.

\begin{lemma}
\label{p ne deli d} For an arbitrary divisor $d$ of $n$,
the following inequality holds:
$$\omega (X_n (1, d)) \leqslant \min_{p\in Q} p = q.$$
\end{lemma}

\begin{proof}
Let $p$ be an arbitrary prime number of $n$ which does not divide
$d$. If we assume that maximal clique has more than $p$ vertices,
then there must be two vertices $a$ and $b$ with the same residue
modulo $p$. This means that $gcd (a - b, n)$ is divisible by $p$,
and therefore is equal to neither $1$ nor $d$. Therefore, we have
$\omega (X_n (1, d)) \leqslant p.$
\end{proof}

\begin{thm}
\label{klika p1pi} If $Q$ is an empty set or $q > p_1 \cdot p_i$,
then
$$
\omega (X_n (1, d)) = p_1 \cdot p_i.
$$
\end{thm}

\begin{proof}
According to Lemma \ref{inequality p1pi} it is enough to construct a
clique of size $p_1 \cdot p_i$ with vertices $x_{rs} = a_s \cdot d +
r$, where $0 \leqslant s < p_i$ and $0 \leqslant r < p_1$. We choose
numbers $a_s$ as solutions of the following congruence equations:
\begin{flalign*}
a_s &\equiv s \quad \ \ \pmod p \quad \mbox { for every } p \in R \\
a_s \cdot d &\equiv s \cdot p_1 \pmod p \quad \mbox { for every } p
\in Q
\end{flalign*}

This linear congruence system has a solution if and only if
$gcd(d,p)\mid s\cdot p_1$ for $p\in Q$. The last relation is
trivially satisfied since $d$ and $p\in Q$ are relatively prime.
Therefore, using Chinese reminder theorem we can uniquely
determinate numbers $a_s$ modulo $M$.

Consider an arbitrary difference $\Delta = x_{rs} - x_{r's'} = d
\cdot (a_s - a_{s'}) + (r - r')$. Assume first that $r \neq r'$. For
every prime divisor $p$ of $d$ (and therefore for every prime $p \in
R$), the number $\Delta$ cannot be divisible by $p$ because $0 < |r
- r'| < p_1$. If $p \in Q$, we have $x_{rs} - x_{r's'} \equiv (s - s') \cdot p_1 + (r - r')
\pmod p$. The residue $|(s - s') \cdot p_1 + (r - r')| \leqslant
(p_i - 1) \cdot p_1 + (p_1 - 1) < p_i \cdot p_1 < q$ is less than
$p$ and never equal to zero - which means that the greatest common
divisor of $\Delta$ and $n$ equals $1$. In other case, we have $r =
r'$. Again, for an arbitrary prime $p \in Q$, the residue of
$\Delta$ modulo $p$ is $(s - s') \cdot p_1$, which is never equal to
$0$. When $p$ is a prime number from $R$, by definition the
difference $a_s - a_{s'}\equiv (s-s') \pmod p$ can not be divisible
by $p$ according to $0 < |s - s'| < p_i \leqslant p$. Therefore, in
this case we have $gcd (\Delta, n) = d$. \rz

When $Q$ is an empty set, we can use the same construction to get a
clique of size $p_1 \cdot p_i$.
\end{proof}

\begin{thm}
\label{klika q} If $q < p_1 \cdot p_i$, then
$$
\omega (X_n (1, d)) = q.
$$
\end{thm}

\begin{proof}
According to Lemma \ref{klika q} it is enough to find $q$ vertices
that form clique in the graph $X_n (1, d)$. Define numbers $x_k =
a_k \cdot d + b_k$ for $k = 0, 1, \ldots, q - 1$, where $b_k$ is the
residue of $k$ modulo $p_1$ and the following conditions are
satisfied:
\begin{flalign*}
a_k \cdot d + b_k &\equiv k \quad \ \ \ \ \pmod p \quad \quad \mbox
{ for every }
p \in Q \\
a_k &\equiv \lfloor k / p_1 \rfloor \pmod p \quad \quad \mbox { for
every } p \in R
\end{flalign*}

Numbers $a_k$ can be uniquely determined using Chinese Reminder
Theorem modulo $M$, because $d$ and $p$ are relatively prime, for
every prime number from $Q$. We will prove that the greatest common
divisor of $x_k - x_{k'}$ and $n$ is always equal to $d$ or $1$, which
would complete the proof. For every $p \in Q$, we have that $x_k -
x_{k'}\equiv k-k'\pmod p$. Since $|k - k'| < q$, $x_k - x_{k'}$ can
not be divisible by $p$. \rz

Next, consider the case when $k$ and $k'$ have the same residue
modulo $p_1$ and $k \neq k'$. If $a_k - a_{k'}$ is divisible by some
$p \in R$, this means that we have also $\lfloor k / p_1 \rfloor
\equiv \lfloor k' / p_1 \rfloor \pmod p$. Since $k$ is less than
$p_1 \cdot p_i$, we conclude that the integer parts of numbers
$\frac{k}{p_1}$ and $\frac{k'}{p_1}$ are equal. Together with the
assumption that fractional parts of these numbers are the same, we
get that $k = k'$. This is a contradiction and therefore $gcd (x_k -
x_{k'}, n) = d$. In the second case, we have that number $x_k -
x_{k'}$ is not divisible by any $p \in R$, because $d$ is
divisible by $p$ and $0 < |b_k - b_{k'}| < p_1$. Thus, we have
$gcd (x_k - x_{k'}, n) = 1$ which completes the proof.
\end{proof}

Finally we reach the following main result of this subsection:

\begin{thm}
\label{main k=2} For any divisor $d$ of $n$, there holds:
$$\omega (X_n (1, d)) =  \min \ \left (
\min_{p \in Q} p, f \left( n \right) \cdot f \left ( \frac
{n}{d}\right) \right).$$
\end{thm}

\subsection{ Case $1 \not \in D$ }

\begin{thm}
Let $X_n (d_1, d_2)$ be a gcd-graph with both divisors
greater than one. Then the following equality holds:
\begin{eqnarray*}
\omega \left( X_n (d_1, d_2) \right) = \left\{
\begin{array}{rl}
\omega (X_{\frac{n}{d_2}} (1, \frac{d_1}{d_2})), & \mbox{ if $d_2 \mid d_1$, } \\
\max \left ( f \left( \frac {n}{d_1}
\right), f \left(\frac {n}{d_2} \right) \right ), & \mbox{ otherwise. }\\
\end{array} \right.
\end{eqnarray*}
\end{thm}

\begin{proof}
If a maximal clique has edges of both colors, then there exists a
non-monochromatic triangle. Therefore, we can find vertices $a$,
$b$, $c$ such that:
$$
gcd(a - b, n) = d_1, \quad gcd(a - c, n) = d_2, \quad gcd(b - c, n) = d_2
$$
By subtraction, we get that $d_2 \mid (a - c) - (b - c)$ and finally
$d_2$ divides $d_1$. We excluded the case with two blue edges,
because than $d_1$ would divide $d_2$, which is impossible.
Therefore, the graph $X_n$ is disconnected according to Theorem~\ref{d components} and we obtained an
equivalent problem for the divisor set $D' = \{1, \frac{d_1}{d_2} \}$
and gcd-graph $X_{n / d_2}(D)$. \rz

In the other case ($d_2$ does not divide $d_1$), the maximal clique is monochromatic
and by Theorem \ref{divisorOne} we completely determine $\omega (X_n (d_1, d_2))$.
\end{proof}

\section{Counterexamples}

In order to test the conjecture proposed in \cite{klotz07} for
integral circulant graphs with more then two divisors, we
implemented \emph{Backtrack Algorithm with pruning} \cite{www} for
finding the clique number. For $k = 3$ and $k = 4$, we construct
infinite families of integral circulant graphs, such that clique
number does not divide $n$. For example, we obtain that
$\omega(X_{20} (1, 4, 10)) = 6$ and $\omega (X_{30} (1, 2, 6, 15)) =
7$ which is verified by an exhausted search algorithm. The next
proposition in theoretic way disproves the conjecture.

\begin{pro}
The clique number of integral circulant graph $X_{20} (1, 4, 10)$
equals $6$ or $7$.
\end{pro}

\begin{proof}
The vertices $0, 1, 4, 8, 11, 12$ form a clique in considered graph,
which is easy to check. We color edges in $X_{20} (1, 4, 10)$ with
three colors: \emph{red} if $gcd(a - b, 20) = 1$, \emph{blue} if
$gcd(a - b, 20) = 4$, and \emph{green} if $gcd(a - b, 20) = 10$.\rz

By Theorem \ref{divisorOne} and Lemma \ref{unconnected unitary}, the
maximal clique with red edges has $2$ vertices, the maximal clique
with blue edges has $5$ vertices, and for green color it has $2$
vertices. If a triangle has one blue and one green edge - it follows
that the third edge cannot be red, because of parity. If the third
edge is blue, than the absolute value of the difference on the green
edge is divisible by $4$, which is impossible. Likewise, if the
third edge is green than the absolute value of the difference on the
blue edge is divisible by $5$. Therefore, there is no triangle in
graph which contain both blue and green edges. \rz

Assume that the maximal clique is two-colored. By previous
consideration maximal clique can contain red and blue edges or red
and green edges. In the first case our problem is to find the
maximal clique in integral circulant graph $X_{20} (1, 4)$. Applying
Theorem \ref{klika q} we conclude that $\omega(X_{20}(1,4))=5$, but
we already found a clique of size $6$. Analogously, the size of
maximal clique with red and green edges is $\omega(X_{20} (1,10)) =
4$ by Theorem \ref{klika p1pi}. It means that maximal clique must
contain all three colors. \rz

Now, the maximal clique is three-colored and consists of $x$ cliques
with blue edges and $y$ cliques with green edges. Using mentioned
fact that there is no triangle with blue and green edges, we can
easily notice that only red edges join these $x + y$ cliques. If we
choose one vertex from each clique, we obtain a red edge clique with
$x + y$ vertices. But, the maximal clique with red edges has only
two vertices, implying that $x = y = 1$. So, the upper bound for the
clique number is $2 + 5 = 7$ and the lower bound is $6$, and neither
of them is a divisor of $20$.
\end{proof}

\begin{pro}
Let $X_n(D)$ be the integral circulant graph with the set of
divisors $D=\{d_1,d_2,\ldots, d_k\}$. If $N=n\cdot p$ where $p$ is
an arbitrary prime number greater then $n$, then the following
equality holds
$$
\omega(X_N(D))=\omega(X_n(D)).
$$
\end{pro}
\begin{proof}
Since $p > n$ and $gcd (a - b, n) = gcd (a - b, p \cdot n)$ for
arbitrary vertices $a, b \in X_n (D)$, we have inequality
$\omega(X_N(D))\geqslant\omega(X_n(D))$. Now, assume that vertices
$\{a_1, a_2, \ldots, a_c\}$ form a maximal clique in $X_N (D)$ and
consider vertices $\{b_1, b_2, \ldots, b_c\}$ in graph $X_n (D)$,
where $b_i$ is the remainder of $a_i$ modulo $n$. Prime number $p$
cannot divide any of the numbers $d_i$ and thus $a_i-a_j$ is not
divisible by $p$ for every $1\leqslant i<j\leqslant c$. Now, we have
$$
gcd (b_i - b_j, n) = gcd (a_i - a_j, n) = gcd (a_i - a_j,N)\in D.
$$
This means that $\omega(X_N(D))\leqslant\omega(X_n(D))$ and finally
$\omega(X_N(D)) = \omega(X_n(D))$.
\end{proof}

Using this proposition we obtain a class of counterexamples $X_{20
 p} (1, 4, 10)$ based on the graph $X_{20} (1, 4, 10)$, where
$p$ is a prime number greater then $20$.

\section{Concluding remarks}

In this paper we moved a step towards describing the clique number
of integral circulant graphs. We find an explicit formula for the
clique number of $X (d_1, d_2, \ldots, d_k)$ when $k \leqslant 2$.
This leads us to the main result of the paper - discarding a
conjecture proposed in \cite{klotz07}, that the clique number
divides $n$. We constructed families of counterexamples for $k = 3$
and $k = 4$. \rz

We also examine numerous examples for $k \geqslant 2$ divisors and
obtain the following inequality:
$$\max_{d_i \in D} \ f \left ( \frac{n}{d_i} \right ) \leqslant \omega (X_n (D))
\leqslant \prod_{i=1}^k f \left ( \frac {n} {d_i} \right ).$$

The lower bound follows from Lemma \ref{unconnected unitary}, and
the upper bound can be proven by induction on $k$. Namely, after
adding edges $\{a, b\}$ such that $gcd (a - b, n) = d_i$ into the
graph $X_n (d_1, d_2, \ldots, d_{i - 1})$, we divide every color
class in maximum $f (\frac{n}{d_i})$ independent parts. Therefore,
the number of color classes is less than or equal to the product of
numbers $f (\frac{n}{d_i})$ for all $d_i \in D$. \rz

We leave for future study to see whether this bound can be improved.
\rz

{\bf Acknowledgement. } The authors are grateful to anonymous referee
for comments and suggestions, which were helpful in improving the manuscript.

\end{document}